\documentclass{amsart}
\usepackage{mathtools,amsmath,amssymb,mathtools,amsthm,overpic,graphicx,calligra,color,geometry}
\usepackage{enumitem}
\usepackage{mathrsfs}
\usepackage{comment}
\usepackage{hyperref}
\usepackage{extarrows}
\usepackage[backend=biber, style=alphabetic]{biblatex}
\usepackage{geometry}
\newtheorem{thm}{Theorem}
\newtheorem*{result}{Result}

\newtheorem{lem}[thm]{Lemma}
\newtheorem{prop}[thm]{Proposition}
\theoremstyle{definition}
\newtheorem{definition}[thm]{Definition}
\theoremstyle{definition}
\begin{document}
\title{Meromorphic convexity on complex manifolds}
\author{Blake J. Boudreaux}
\address{Department of Mathematical Sciences, University of Arkansas, Fayetteville, AR 72701, USA}
\email{bb225@uark.edu}
\author{Rasul Shafikov}
\address{Department of Mathematics, University of Western Ontario,  London, Ontario N6A 5B7, Canada}
\email{shafikov@uwo.ca}
\date{\today}

\begin{abstract}
The notion of meromorphic convexity is defined and studied on complex manifolds. Using this notion,
in analogy with Stein manifolds, a new class of complex manifolds, called {\calligra M }-manifolds, is introduced.
This is a class of complex manifolds with a good supply of global meromorphic functions, in particular, it includes
all Stein manifolds and projective manifolds. It is also shown that there exist noncompact complex manifolds,
known as long $\mathbb C^2$, that are {\calligra M }-manifolds but do not contain any nonconstant holomorphic
functions.
\end{abstract}
\maketitle

\section{Introduction}
Holomorphic convexity is an important property of complex manifolds, and one of the defining characteristics of Stein manifolds.
Holomorphically convex manifolds are similar to Stein manifolds, with the exception that they are allowed to admit compact analytic varieties of positive dimension. In fact, through the {\it Remmert reduction} there is a unique, up to isomorphism, Stein space $Y$ that can be associated with a holomorphically convex manifold $X$. The Stein space $Y$ has the property that $\mathscr{O}(X)\cong\mathscr{O}(Y)$, and so holomorphic function theory on $X$ reduces to the holomorphic function theory on a Stein space $Y$ by passing through the quotient map. In particular, this can be used to prove the Oka--Weil theorem on holomorphically convex complex manifolds~\cite[Theorem 3.2]{Mo20}.

The purpose of this paper is to initiate the development of an analogous theory for \textit{meromorphic} functions on a complex manifold $X$. Let $\mathscr{O}_X$ denote the sheaf of germs of holomorphic functions on $X$. Through a formal algebraic construction, $\mathscr{O}_X$ gives rise to $\mathscr{M}_X$, the sheaf of quotients of $\mathscr{O}_X$, called the \textit{sheaf of germs of meromorphic functions on $X$}. A \textit{meromorphic function} on an open set $\Omega\subseteq X$ is defined to be a section of $\mathscr{M}_X$ on $\Omega$, and the space of meromorphic functions on an open set $\Omega$ is denoted $\mathscr{M}(\Omega)$. By construction, it follows that, given $f\in\mathscr{M}(\Omega)$, every point of $\Omega$ admits a neighborhood $U$ on which $f=u/v$, where $u,v\in\mathscr{O}(U)$ and $\gcd(u,v)=1$. If $X$ is Stein, then any $f\in\mathscr{M}(X)$ admits a representation $f=u/v$ for globally defined $u,v\in\mathscr{O}(X)$, and we can further demand that $\gcd(u,v)=1$ if and only if $X$ additionally satisfies the topological condition $H^2(X,\mathbb{Z})=\{0\}$, see~\cite{Ep78} and~\cite{BoSh25} for further details. However,
no such global representation exists in general.

To every nonzero meromorphic function $f$ we can associate a divisor of zeroes and a divisor of poles, denoted by $Z(f)$ and $P(f)$, respectively. The divisors $Z(f)$ and $P(f)$ are precisely those with support $\overline{\{f=0\}}$ and $\overline{\{1/f=0\}}$, respectively, and with coefficients determined by the multiplicities of $f$ and $1/f$, respectively. Indeed, in view of~\cite[Theorem 1.4.4]{Ch89}, the closures of the complex analytic sets $\{f=0\}$ and $\{1/f=0\}$ extend to complex hypersurfaces on the complex manifold.

If $f$ is a meromorphic function on a complex manifold $X$, then the set $Z(f)\cap P(f)$ forms a complex analytic set of complex codimension at least two. When $\dim_{\mathbb{C}}(X)=1$, this means that $Z(f)\cap P(f)$ is empty, and so $f$ can be viewed as a holomorphic mapping from $X$ into $\mathbb{CP}^1$. On the other hand, if $\dim_{\mathbb{C}}(X)>1$, then $f$ has no well-defined value at points of $Z(f)\cap P(f)$. The set $Z(f)\cap P(f)$ is called the \textit{set of indeterminacy points of $f$} and is denoted $\mathcal{I}(f)$. From this point of view, meromorphic functions in higher dimensions are truly meromorphic objects.

In Section~\ref{s.mc} we give general properties of meromorphically convex hulls and define the notion of meromorphic convexity for
complex manifolds. In Section~\ref{s.mm} we introduce a new class of complex manifolds---dubbed {\calligra M }-manifolds---which should be considered as a meromorphic analogue of Stein manifolds. Section~\ref{s.ow} discusses some variations of the classical Oka--Weil theorem.
One of the principal results of the paper, the existence of long $\mathbb C^2$ that are {\calligra M }-manifolds, is the content of Section~\ref{s.longc2}. Combined with the work of Boc Thaler and Forstneri\v c~\cite{BTFo16}, this gives an example of an {\calligra M }-manifold that
contains no nonconstant holomorphic functions. Finally, in the last section we prove some additional results concerning meromorphic functions on certain holomorphically and meromorphically convex manifolds.

Throughout this paper we assume that all manifolds are second countable, and connected unless otherwise specified.

\smallskip

\noindent{\bf Acknowledgments.} We would like to thank the anonymous referee for valuable comments which have improved the quality and exposition of the paper. The second author is partially supported by Natural Sciences and Engineering Research Council of Canada.

\section{Meromorphic Convexity}\label{s.mc}

The notion of meromorphic convexity generalizes that of rational convexity on $\mathbb C^n$ and the
complex projective space $\mathbb{CP}^n$. In this section $X$ is an arbitrary complex manifold.

\begin{definition}\label{MeroCvx}
Given a compact set  $K\subseteq X$, define
\begin{equation}\label{hull}
	\widehat{K}_X=\left\{z\in X: |f(z)|\leq\|f\|_K \text{ for every } f\in\mathscr{M}(X)\cap\mathscr{O}(K\cup \{z\}) \right\}.
\end{equation}
\end{definition}

Note that the space $\mathscr{M}(X)\cap\mathscr{O}(K\cup \{z\})$ is never empty as it contains constant functions.
We call $\widehat{K}_X$ the \textit{meromorphically convex hull} of $K$.
When there is no chance of confusion, the subscript on the hull may be omitted. Note that in the literature, $\widehat K$ often denotes
the polynomially or holomorphically convex hull of~$K$. For convenience, in this paper
we use this notation for meromorphically convex hulls.

\noindent\textbf{Remark.} The meromorphically convex hull of $K$, as defined in Definition~\ref{hull}, differs formally from the definition given in~\cite{BoSh25}, where it is described as
\[
	\mathscr{M}\text{-hull}(K)=\left\{z\in X:|f(z)|\leq \|f\|_K\text{ for all $f\in\mathscr{M}(X)$ with $\mathcal{I}(f)\cap (K\cup\{z\})=\varnothing$}\right\}.
\]
Nonetheless, for every compact set $K$, one has $\mathscr{M}\text{-hull}(K)=\widehat{K}_X$, and hence the two notions coincide. Indeed, the inclusion $\mathscr{M}\text{-hull}(K)\subseteq\widehat{K}_X$ is immediate.

For converse, suppose $z\not\in\mathscr{M}\text{-hull}(K)$. Then there exists a meromorphic function $f$ such that
\[
	\mathcal{I}(f)\cap (K\cup\{z\})=\varnothing\qquad\text{and}\qquad |f(z)|>\|f\|_K.
\]
If $z\not\in P(f)$, then $f\in\mathscr{M}(X)\cap\mathscr{O}(K\cup\{z\})$, and it follows that $z\not\in\widehat{K}_X$.

If instead $z\in P(f)$, one can compose $f$ with a M\"obius transformation $T:\mathbb{CP}^1\to\mathbb{CP}^1$ that is sufficiently close to the identity and satisfies $|T(\infty)|<\infty$. Then $T\circ f$ is a meromorphic function which is holomorphic on $K\cup\{z\}$ and still satisfies $|T(f(z))|>\|T\circ f\|_K$, reducing the argument to the previous case.

The modification of the definition is thus purely cosmetic and serves only to streamline the proofs.

We also identify the following set associated with the meromorphically convex hull.

\begin{definition}\label{InnerHull}
Let
\begin{equation}\label{innerhull}
	\widetilde{K}_X=\left\{z\in X:\text{for every $f\in\mathscr{M}(X)\cap\mathscr{O}(K)$
	it follows that $f\in\mathscr{O}_z$ and $|f(z)|\leq\|f\|_K$}\right\}.
\end{equation}
We call $\widetilde{K}_X$ the \textit{inner hull} of $K$.
\end{definition}

\noindent Clearly, $\widetilde K_X \subseteq \widehat K_X$, with equality holding for all known examples.
However, without any additional information about the space of meromorphic functions we cannot establish the identity
$\widetilde K_X = \widehat K_X$ for all complex manifolds. It is immediate that, on $X=\mathbb C^n$, the meromorphically convex
hull for any compact set $K$ agrees with its rationally convex hull, which also agrees with the convex hull of $K$ with respect to
complex hypersurfaces, which is defined as follows (cf.~\cite{BoSh25,Co99,Hi71}).

\begin{definition}
	Let $X$ be a complex manifold and let $K\subset X$ be a compact set. Define
\[
h(K)=\{z\in X:\text{ every complex hypersurface in $X$ passing through $z$ intersects $K$}\}.
\]
We say that $K$ is \textit{convex with respect to hypersurfaces} if $h(K)=K$, and that $X$ is \textit{convex with respect to hypersurfaces}
if $h(K)$ is compact whenever $K$ is compact.
\end{definition}

\noindent  We now collect some general properties of meromorphic  hulls and inner hulls on complex manifolds.

\begin{prop}\label{BasicProps}
	Let $X$ be a complex manifold, and $K\subset X$ be a compact set. Then
	\begin{enumerate}
		\item[(i)] $\widehat K_X$ is a closed set.
		\item[(ii)] $\widetilde K_X  \subseteq \widehat K_X$, and
			$(\widehat K_X)^\circ  \subseteq\widetilde K_X$.
		\item[(iii)] $f\in\mathscr{O}(\widetilde K_X)$ whenever $f\in\mathscr{M}(X)\cap\mathscr{O}(K)$.
		\item[(iv)] $h(K)\subseteq\widetilde{K}_X$.
	\end{enumerate}
\end{prop}

\begin{proof}
(i) If any meromorphic function on $X$ that is holomorphic on $K$ is constant, then $\widehat K_X = X$ and there is nothing to prove. Otherwise, suppose
$p \in X \setminus \widehat K_X$. If for every nonconstant $f\in \mathscr{M}(X) \cap \mathscr O(K)$ the point $p$ were either a pole or
an indeterminacy of $f$, then $p$ would be in $\widehat K_X$. So there exists a nonconstant
$f \in \mathscr{M}(X) \cap \mathscr O(K\cup \{p\})$, and
we have $|f(p)| > ||f||_K$. This shows that the complement of $\widehat K_X$ is open.

(ii) The inclusion $\widetilde K_X  \subseteq \widehat K_X$ is obvious. To prove the second inclusion, let
 $p \in (\widehat K_X)^\circ$. Then for any $f \in \mathscr M(X) \cap \mathscr O(K)$ we have $f \in \mathscr O_p$. Indeed, suppose first that
	$p\not\in\mathcal{I}(f)$ is a pole for $f$. Then there exist $\zeta \in \mathbb C \setminus f(K)$ and a linear-fractional transformation $\sigma$ of $\mathbb{CP}^1$ such that
$\sigma (\zeta) = \infty \in \mathbb{CP}^1$ and $|\sigma (\infty) | > ||\sigma||_{f(K)}$. The meromorphic function $\tilde f = \sigma \circ f$
is holomorphic on $\mathscr O(K\cup \{p\})$ and satisfies $|\tilde f(p)| > ||\tilde f||_K$. But this contradicts $p \in \widehat K_X$.
The remaining possibility is that~$p$ is an indeterminacy point of~$f$. In this case, since $p$ is an interior point of $\widehat K_X$, there exists a point
$q$ sufficiently close to $p$ which is a pole of $f$, and we obtain a contradiction as above. This shows that $f \in \mathscr M(X) \cap \mathscr O(K)$
implies $f \in \mathscr O_p$. Then by the definition of meromorphic convexity we have $|f(p)| \le ||f||_K$ and therefore, $p \in \widetilde K_X$.

(iii) This follows from the definition of inner hull.

(iv) Let $z\in X\setminus\widetilde{K}_X$. Then either
there exists $f\in\mathscr{M}(X)\cap\mathscr{O}(K)$ with $f\not\in\mathscr{O}_z$ or $f \in \mathscr{O}_z$ and $|f(z)|>\|f\|_K$. If $f\not\in\mathscr{O}_z$, then $z\in P(f)$, the divisor of poles of the function $f$. Since $f\in\mathscr{O}(K)$, $P(f)$ cannot intersect $K$ and it follows that $P(f)$ is the desired hypersurface. If $f\in\mathscr{O}_z$ and $|f(z)|>\|f\|_K$, then the hypersurface $\overline{f^{-1}(f(z))}$ suffices.
\end{proof}

The following proposition describes the extension property of  $\widetilde{K}_X$ on general complex manifolds.
For a compact set $K$, we
denote by $\overline{\mathscr{M}(X)}_{\mathcal{C}(K)}$ the closure, in the uniform norm on $K$, of the space of meromorphic functions on $X$ whose pole divisor does not intersect $K$.

\begin{prop}\label{p.ext} Let $X$ be a complex manifold and let $K\subseteq X$ be compact.
    \begin{enumerate}
	\item[(i)] The inner hull $\widetilde{K}_X$ is the largest set to which all $f\in\mathscr{M}(X)\cap\mathscr{O}(K)$ extend holomorphically.
	\item[(ii)] Every $f\in\overline{\mathscr{M}(X)}_{C(K)}$ extends naturally to a unique function $\widetilde{f}\in\overline{\mathscr{M}(X)}_{\mathcal{C}(\widetilde{K}_X)}$ that satisfies $||\widetilde f||_{\widetilde K_X} = ||f||_K$.
    \end{enumerate}
\end{prop}
\begin{proof}
	(i) Let $L$ be the largest set in $X$ with the property that $f\in\mathscr{O}(L)$ for all $f\in\mathscr{M}(X)\cap\mathscr{O}(K)$. By
	Proposition~\ref{BasicProps}(iii), we see that $\widetilde{K}\subseteq L$. If $z\not\in\widetilde{K}$, then there exists a
	$g\in\mathscr{M}(X)\cap\mathscr{O}(K)$ with $g\not\in\mathscr{O}_z$ or $|g(z)|> \|g\|_K$. In the former case, we see that $z\not\in L$.
	In the latter case, we can assume $g\in\mathscr{O}_z$. If we set $p=g(z)$, then the assignment
\[
	w\mapsto\frac{1}{g(w)-p}
\]
	produces a member $h\in\mathscr{M}(X)\cap\mathscr{O}(K)$ which does not extend holomorphically to a neighborhood of $z$. We conclude that $\widetilde{K}=L$.

	(ii) If $f\in\overline{\mathscr{M}(X)}_{\mathcal C(K)}$, then there exists is a sequence $\{f_j\}\subseteq\mathscr{M}(X)\cap \mathscr O(K)$ converging uniformly to $f$ on $K$. Pick a point $z\in\widetilde K_X$. Then $f_j \in \mathscr O_z$,  and $|f_j(z)|\leq \|f_j\|_K$, hence $\{f_j(z)\}\subseteq\mathbb{C}$ is a Cauchy sequence. Set $\widetilde{f}(z)$ to be the limit of this sequence. By uniform convergence, $\widetilde{f}$ is continuous and we see that
	$\widetilde{f}\in\mathscr{M}(K)$. This also proves the equality of the two norms.
\end{proof}

Next, we define meromorphically convex manifolds.

\begin{definition}\label{MeroCvxM}
A complex manifold $X$ is called \textit{meromorphically convex} if
$\widehat{K}_X\subseteq X$ is compact for every compact set $K\subseteq X$.
\end{definition}

Clearly, $\mathbb C^n$ is meromorphically convex with the meromorphically convex hull and the inner hull both being equal to the rationally convex hull for all compact subsets. The next proposition gives basic properties of the hulls on meromorphically convex manifolds.

\begin{prop}\label{MerCvxProp}
Let $X$ and $Y$ be meromorphically convex complex manifolds, and let $K\subset X$ be a compact set. Then
	\begin{enumerate}[resume]
		\item [(i)] If $\widetilde K = \widehat K$, then $\widehat{(\widehat K_X)}_X = \widehat K_X$.
		\item[(ii)] If $X$ and $Y$ are two open manifolds in some ambient complex manifold, and $X\cap Y$ connected,
			then $X \cap Y$ is meromorphically convex.
		\item[(iii)] $ X\times Y$ is meromorphically convex.
\end{enumerate}
\end{prop}

\begin{proof}
(i) The inclusion $\widehat{K}_X\subseteq\widehat{(\widehat{K}_X)}_X$ is trivial.

For the opposite inclusion, let $z \in \widehat{(\widehat K_X)}_X$ be an arbitrary point.
Arguing by contradiction, suppose that $z\notin \widehat K_X$. This implies that there exists  a function
$f\in \mathcal M(X) \cap \mathcal O (K)$ such that $f \in \mathcal O_z$ and $|f(z)| > ||f||_K$.
Since $\widetilde K_X = \widehat K_X$,  by Proposition~\ref{BasicProps}(iii),, any $f \in \mathcal M(X) \cap \mathcal O (K)$
extends holomorphically to $\widehat K$ with $||f||_K = ||f||_{\widehat K}$. It follows that $|f(z)| > ||f||_{\widehat K}$,
which means that $z \notin \widehat{(\widehat K_X)}_X$. This contradiction proves $\widehat{\widehat{K}} \subset \widehat{K}$.

(ii) For any compact set $K \subseteq X \cap Y$, the set $\widehat K_{X\cap Y}$ is closed by Proposition~\ref{BasicProps}(i). We claim that
$\widehat K_{X\cap Y} \subseteq \widehat K_X$. Indeed, if $z \in (X\cap Y) \setminus\widehat K_X$, then there exists $f \in \mathscr M(X) \cap \mathscr O(K\cup \{z\})$ such that
	$|f(z)|> ||f||_K$. By viewing $f$ as a meromorphic function on $X \cap Y$, we see that $z \notin \widehat K_{X\cap Y}$.
Thus,
$\widehat K_{X\cap Y} \subset [\widehat K_X \cap (X\cap Y) ]$.
Similarly, $\widehat K_{X\cap Y} \subset [\widehat K_Y \cap (X\cap Y) ]$, and therefore,
$$
\widehat K_{X\cap Y} \subset \widehat K_X \cap \widehat K_Y.
$$
	By Proposition~\ref{BasicProps}(i), $\widehat{K}_{X\cap Y}$ is thus a compact subset of $\widehat{K}_X\cap\widehat{K}_Y$. Since $\widehat{K}_{X\cap Y}\subset X\cap Y$ by definition, $\widehat{K}_{X\cap Y}$ is also a compact subset of $X\cap Y$.

(iii) Let $K\subset X \times Y$ be a compact set. Let $\pi_X : X\times Y \to X$ and $\pi_Y : X\times Y \to Y$ be the natural projections, and let
$K_1 = \pi_X (K)$ and $K_2 = \pi_Y(K)$. Then it is easy to see that
$$
	\widehat K_{X\times Y} \subseteq \widehat{(K_1)}_X \times  \widehat{(K_2)}_Y.
$$
The set on the right-hand side of the above inclusion is compact, and since $\widehat K_{X\times Y}$ is closed, it is a compact subset of $X\times Y$.
\end{proof}

There are plenty of examples of meromorphically convex manifolds. Since every holomorphic function is meromorphic, any holomorphically convex complex manifold is meromorphically convex, in particular, any Stein manifold is meromorphically convex. All compact complex manifolds are trivially meromorphically convex, in particular all projective manifolds are meromorphically convex. Further, the Cartesian product of a Stein manifold and a projective manifold is holomorphically convex, hence meromorphically convex.

Our next goal is to give some additional properties of the inner hulls for certain manifolds.

\begin{prop}\label{steinProject} The following holds
\begin{enumerate}
    \item[(i)] On any Stein manifold $X$, we have $\widehat K_X=\widetilde{K}_X= h(K)$ for any compact $K \subset X$.
\item[(ii)] For any compact $K \subset \mathbb{CP}^n$, $\widetilde K_X =\widehat K_X$, which also agrees with the rationally convex hull of $K$ in
$\mathbb{CP}^n$.
\item[(iii)] Let $X \subset \mathbb{CP}^N$ be a projective manifold and let $K\subset X$ be a compact set which omits an positive divisor; i.e., there exists an ample holomorphic line bundle $A$ and a $\sigma\in H^0(X,A)$ with $\sigma^{-1}(0)\cap K=\varnothing$. Then $\widetilde K=\widehat K$
\end{enumerate}
\end{prop}

\begin{proof}
	(i) By Proposition~\ref{BasicProps} we already have the inclusion $h(K)\subset \widetilde K \subset \widehat K$.
Since $X$ is Stein, by Proposition~1.2 of~\cite{BoSh25} we have $\widehat{K}=h(K)$, which proves the required statement.

(ii) This follows from the fact that on $\mathbb{CP}^n$ meromorphic functions are precisely rational functions and rational convexity is equivalent to the convexity with respect to complex hypersurfaces or with respect to positive divisors~\cite[Lemma 2.2]{Gu99}.

(iii) It follows from Proposition~\ref{BasicProps}(ii) that $\widetilde K \subset \widehat K$. To prove the reverse inclusion, suppose that $p \in X \setminus \widetilde K$. Then there exists a function $f \in \mathcal M(X) \cap \mathcal O(K)$ such that $p$ is either a pole or an indeterminacy point of $f$. If $p$ is a pole, by composing $f$ with a M\"obius transformation of $\mathbb{CP}^1$ which is sufficiently close to the identity, we see that $p\in X\setminus \widehat K$. Suppose now that $p$ is an indeterminacy point of $f$. Then the result immediately follows from the following lemma, which is also of independent interest.

\begin{lem}\label{l.g-p}
Let $X \subset \mathbb{CP}^N$ be a projective manifold and let $K \subset X$ be a compact subset which omits a positive divisor. Suppose $f \in \mathcal M(X) \cap \mathcal O(K)$ and $p\in X\setminus K$ is a point of indeterminacy for $f$. Then, there exists $g \in \mathcal M(X) \cap \mathcal O(K)$ such that $p$ is a genuine pole of $g$ (i.e., $g(p) = \infty$).
\end{lem}

\begin{proof}[Proof of Lemma~\ref{l.g-p}]
    Since the zero and pole divisors of $f$ are effective, there exist holomorphic line bundles $L_0\to X$ and $L_1\to X$ as well as holomorphic sections $s_0\in H^0(X,L_0)$ and $s_1\in H^0(X,L_1)$ such that $\text{div}(s_0)=P(f)$ and $\text{div}(s_1)=Z(f)$. Since $f$ and $s_1/s_0$ have the same divisor and $f$ is in particular a meromorphic section of the trivial bundle, it follows that $L_0=L_1=:L$. Furthermore, compactness of $X$ ensures that $s_1/s_0$ is a nonzero scalar multiple of $f$; absorbing this scalar into $s_1$ yields
\[
    f=\frac{s_1}{s_0}.
\]

Recall that $A$ is an ample line bundle and $\sigma$ is a holomorphic section of $A$ whose zero set avoids $K$. By the Kodaira embedding theorem, there exists a $k\in\mathbb{N}$ and a basis $\sigma_{0}=\sigma^k,\sigma_{1},\ldots,\sigma_N$ of $H^0(X,A^k)$ so that the map
\begin{equation*}
    \Phi(x)=[\sigma_{0}(x):\ldots:\sigma_N(x)]\in\mathbb{CP}^N
\end{equation*}
defines a holomorphic embedding of $X$ into a subvariety $Y$ of $\mathbb{CP}^N$ with $A^k=\Phi^*(\mathcal{O}_X(1)|_Y)$. The compact set $\Phi(K)$ thus avoids $\{[y_{0}:\ldots :y_N]\in\mathbb{CP}^N:y_{0}=0\}$, the hyperplane at infinity.

We seek a section $\tau \in H^0(X, A^{\otimes km})$ (for some $m \ge 1$) such that:
$$ \tau(x) \neq 0 \quad \forall x \in K, \quad \text{and} \quad \tau(p) \neq 0. $$
If $\sigma_0(p)\neq 0$, then we can take
$$
\tau = \sigma_0^m.
$$ 
Clearly, $\tau$ vanishes nowhere on $K \cup \{p\}$. The case $y_0(p)=0$ can be handled by a small perturbation of the hyperplane at infinity in $\mathbb{CP}^N$. 

Consider the twisted line bundle $\tilde{L} = L \otimes A^{\otimes km}$. Since $A$ is ample, the Cartan--Serre--Grothendieck Theorem (e.g., Lazarsfeld~\cite[Thm.~1.2.6]{La04}) the twisted bundle $\tilde{L} = L \otimes A^{\otimes km}$ is generated by its global sections for $m$ sufficiently large. We assume that $m$ above was chosen sufficiently large to satisfy this.
It follows that there exists a global section $S_1 \in H^0(X, \tilde{L})$ such that:
$S_1(p) \neq 0$.

We define the new denominator section as $S_0 = s_0 \otimes \tau$. 
Note that $S_0 \in H^0(X, L \otimes A^{\otimes km}) = H^0(X, \tilde{L})$.
We define the function $g$ as:
$$ g = \frac{S_1}{S_0} = \frac{S_1}{s_0 \otimes \tau}. $$
For any $x \in K$, we have $s_0(x) \neq 0$ (by hypothesis) and $\tau(x) \neq 0$ by construction. Thus, $S_0(x) \neq 0$ for all $x \in K$. Consequently, $g$ is holomorphic on $K$. Further,  $S_1(p) \neq 0$, and 
 $S_0(p) = s_0(p)\tau(p) = 0$.
Since the numerator is nonzero and the denominator is zero, $p$ is a pole of $g$. Specifically, $p \notin \mathcal{I}(g)$.
\end{proof}

This completes the proof of Proposition~\ref{steinProject}. 

\end{proof}

\section{{\calligra M }-manifolds}\label{s.mm}

In analogy with Stein manifolds, we consider the following class of complex manifolds.

\begin{definition}\label{M-manifold}
A complex manifold $X$ of dimension $n \ge 1$ is called an {\calligra M }-manifold if the following conditions are satisfied
\begin{enumerate}
    \item[(a)] $X$ is meromorphically convex, i.e, $\widehat K_X$ is a compact subset of $X$ for any compact $K \subset X$;
    \item[(b)] $\mathscr M(X)$ separates points, i.e., for any points $p,q \in X$, $p\ne q$, there exists a meromorphic function $f$ on $X$ such that $f$ is holomorphic near $p$ and $q$, and $f(p)\ne f(q)$;
    \item[(c)] Existence of local coordinates: for any point $p\in X$ there exists a neighborhood $U$ of $p$ and meromorphic functions
	    $f_1,\dots, f_n$ such that $\{f_1|_U, \dots, f_n|_U\}$ form a local holomorphic coordinate system on $U$.
\end{enumerate}
\end{definition}

\noindent The following are basic examples concerning {\calligra M }-manifolds.

\begin{enumerate}

    \item Stein manifolds are {\calligra M }-manifolds. Indeed, on a Stein manifold meromorphic convexity as defined in Definition~\ref{M-manifold} agrees with weak meromorphic convexity defined in~\cite{BoSh25} (see the remark following Definition~\ref{MeroCvx}) which implies (a).  Properties (b) and (c) immediately follow from the Steinness of $X$. 

    \item Let $X$ be the complex manifold obtained by blowing up the origin in the unit ball in $\mathbb C^2$. Then $X$ is holomorphically convex but not Stein. Any holomorphic function on $X$ is constant on the exceptional divisor, and so for any point in the exceptional divisor there are no local coordinates that are formed by entire functions. However, $X$ is an {\calligra M }-manifold.

    \item Projective manifolds are {\calligra M }-manifolds, in particular, projective space is an {\calligra M }-manifold. Meromorphic convexity was shown in the previous section, and properties (b) and (c) follow from the fact that meromorphic functions on a projective manifold $X$ are the restriction of rational functions on the ambient projective space.

    	\item Any Riemann surface is an {\calligra M }-manifold. Indeed, a noncompact Riemann surface is a Stein manifold, and any compact Riemann surface is projective.

    \item A Cartesian product of a Stein manifold and a projective manifold is an  {\calligra M }-manifold, see Proposition~\ref{p.m-m} below.

    \item There exist long $\mathbb C^2$ that are not holomorphically convex. In fact, there exist long $\mathbb{C}^2$ which admit no nonconstant holomorphic functions, yet are {\calligra M }-manifolds, see the next section.

    \item By the Thimm--Siegel--Remmert theorem (see the next paragraph), there exist compact manifolds with no nonconstant meromorphic functions. In particular, there exists a Hopf manifold $\mathscr H$, $\dim_{\mathbb{C}}\mathscr{H}\geq 2$, which contains no nonconstant meromorphic functions. Therefore, $\mathscr H$ is (trivially) meromorphically convex but  $\mathscr H \setminus\{p\}$ is not.  For any Stein (or just meromorphically convex) manifold $X$, the manifold $\mathscr H \times X$ is meromorphically convex but not an {\calligra M }-manifold.
\end{enumerate}

Recall that by the result of Thimm~\cite{Thi}, Siegel~\cite{Si55}, and Remmert~\cite{Re56}, see also Andreotti--Stoll~\cite{AS}, for a compact complex manifold $X$, the meromorphic function
field $\mathscr{M}(X)$, viewed as a field extension over $\mathbb C$,  has transcendence degree $d$ satisfying
$0\le d\le \dim_{\mathbb C} X$. When $d=0$, the only meromorphic functions are constants, this case includes some Hopf manifolds and complex tori. Manifolds for which the transcendence degree of $\mathscr{M}(X)$ is $\dim_{\mathbb C} X$ are called Moishezon manifolds. This class includes all projective manifolds.


\begin{prop}\label{p.moish}
Let $X$ be a compact complex manifold. If $X$ is an {\calligra M }-manifold, then $X$ is Moishezon.
\end{prop}

\begin{proof}
The proof is a well-known argument, see, e.g.,~\cite[Thm 2.2.9]{MaMa}.
Suppose that the transcendence degree of $\mathcal M(X)$ is less than $n=\dim X$. This means that any $n$ meromorphic
functions $f_1, \dots, f_n$ on $X$ are algebraically dependent, i.e., there exists a nontrivial polynomial
$P \in \mathbb C[z_1, \dots, z_n]$
such that
\begin{equation}\label{e.pd}
P(f_1,\dots, f_n)=0 .
\end{equation}
 Without loss of generality we may assume that $f_1, \dots f_{n-1}$ are algebraically independent, and let $P$ be
 a nonzero polynomial of minimal degree in $z_n$ such that $P(f_1,\dots, f_n)=0$. Differentiation yields
 $$
\sum_{j=1}^n \frac{\partial P}{\partial z_j}(f_1,\dots, f_{n}) df_j =0,
$$
on the domain $U$ where all $f_j$ are holomorphic. This shows that the differentials $df_j$ are linearly dependent, i.e.,
$df_1 \wedge \dots \wedge df_{n} (z) = 0 $ for any $z\in U$. On the other hand, if  $X$ is an {\calligra M }-manifold, by
property (c), for any $p\in X$ there exists an open set $U$ and meromorphic functions $f_1, \dots, f_n$ such that
$df_1 \wedge \dots \wedge df_n (z) \ne 0$ for any $z\in U$. This contradiction proves the result.
\end{proof}

It is well-known that all Moishezon manifolds of dimension $n \le 2$ are projective, and therefore, these are {\calligra M }-manifolds. But for $n \ge 3$ there exist nonprojective (and non-K\"ahler) Moishezon manifolds~\cite{Moi,Shafa}. Given a nonprojective Moishezon manifold $X$, there exists a bimeromorphic equivalence to a projective manifold $Y$, which gives an isomorphism between $\mathscr{M}(X)$ and $\mathscr{M}(Y)$. And although the separation property (b) holds for $\mathscr{M}(Y)$ when $Y$ is projective, the isomorphism does not immediately imply that the same holds for $\mathscr{M}(X)$. Hironaka \cite{Hi60,Hi62} was the first to construct a nonprojective three-dimensional Moishezon manifold; the variant relevant to our discussion is \cite[App.~B, Ex.~3.4.2]{Ha77}. Let $Y$ be a smooth projective variety of dimension three and $C \subset Y$ an irreducible curve whose only singularity is a double point $p$ with two distinct tangents. In a small neighborhood of $p$ the two branches of $C$ are blown up in sequence, while over $Y \setminus \{p\}$ the smooth curve $C \setminus \{p\}$ is blown up; the two constructions agree over the punctured neighborhood, where the branches are disjoint, and glue to a compact complex manifold $X$ together with a proper modification $\pi \colon X \to Y$. In particular, $X$ is Moishezon. Note that $X$ is not a blowup of $Y$: since $C$ is irreducible, the local ordering of the branches does not extend globally. The fibre $\pi^{-1}(p)$ consists of two rational curves, and the component $C_0$ arising from the branch blown up first is null-homologous, $[C_0] = 0$ in $H_2(X;\mathbb{Z})$ \cite[Ex.~B.3.4.2]{Ha77}. Now let $g \in \mathscr{M}(X)$ with $C_0 \not\subset\mathcal{I}(g)$, and suppose that $g|_{C_0}$ is nonconstant. Since $C_0 \cong \mathbb{CP}^1$, the map $g|_{C_0}$ is surjective, and for every $c \in \mathbb{CP}^1$ the effective divisor $D_c = \overline{g^{-1}(c)}$ does not contain $C_0$ (otherwise $g|_{C_0} \equiv c$); hence $D_c$ meets $C_0$ in a nonempty finite set, and so $D_c \cdot C_0 > 0$, contradicting $[C_0] = 0$. Therefore, every meromorphic function on $X$ is either constant on $C_0$ or indeterminate along all of $C_0$; in either case, $\mathscr{M}(X)$ does not separate points of $C_0$, and so $X$ is not an $\mathscr{M}$-manifold.

In general it is not known whether every nonprojective Moishezon manifold admits a null-homologous complex curve as discussed above, and so it is an open problem to characterize Moishezon manifolds that are {\calligra M }-manifolds.

We call a complex manifold $X$ {\it meromorphically spreadable} if for any point $p\in X$ there exist meromorphic functions $f_1,\dots, f_N$ on $X$ such that $p$ is an isolated point in the variety $\overline{F^{-1}(F(p))}$, where $F=(f_1,\dots, f_N)$. In the context of Stein manifolds, property (b) or (c) in Def.~\ref{M-manifold} is equivalent to holomorphic spreadability. We do not know if this holds for all {\calligra M }-manifolds, but some partial results are provided in the next proposition.

\begin{prop}\label{p.m-m}
Let $X$ be a
complex
manifold. Then
\begin{enumerate}
	\item[(i)] Property (b) in Definition~\ref{M-manifold} $\Longleftrightarrow$ $\{p\}$ is a meromorphically convex compact for any $p\in X$.
	\item[(ii)]Property (b) $\Longrightarrow$ $X$ is meromorphically spreadable.
	\item[(iii)]Property (c) $\Longrightarrow$ $X$ meromorphically spreadable.
	\item[(iv)]Let $X_1, X_2 \subset X$ be two open {\calligra M }-manifolds in a complex manifold $X$. Then $X_1 \cap X_2$ is an {\calligra M}-manifold.
	\item[(v)] If $X_1$ and $X_2$ are  {\calligra M }-manifolds, then so is $X_1 \times X_2$.
\end{enumerate}
\end{prop}

\begin{proof} $\ $
\begin{enumerate}
    \item[(i)]
Given any $p\in X$, suppose that $q\in \widehat{\{p\}}_X$, $q \ne p$. Since $X$ is meromorphically separable, there exists a meromorphic function $f$ such that $f(p) \ne f(q)$. Let $\tilde f = f - f(p)$. Then $0 =|\tilde f(p)|< |\tilde f(q)|$, which contradicts
$q \in \widehat{\{p\}}_X$. Conversely, if $\widehat{\{p\}}=\{p\}$, and $q \ne p$, then there exists a meromorphic function $f$ on $X$ which is holomorphic
on $\{p,q\}$ and $f(p) \ne f(q)$.

\item[(ii)] Let $p\in X$ be arbitrary, and let $f_1$ be a nonconstant meromorphic function on $X$ which is holomorphic near $p$ (exists by meromorphic separability). Let $V_1$ be the germ at $p$ of the complex hypersurface $\{z \in X : f_1(z) =f_1(p)\}$. Then there exists a meromorphic function $f_2$ on $X$, also holomorphic near~$p$, such that $f_2|_{V_1}\ne\text{ const}$. Then the germ $V_2$
at $p$ of the variety $\{z\in X : (f_1,f_2)(z) = (f_1, f_2)(p)\}$ is smaller than $V_1$.  We may repeat this process, at each step adding a function $f_j \in \mathcal M(X) \cap \mathcal O_p$ so that the germ $V_j=\{z\in X : f_\nu(z)=f_\nu(p), \ \nu=1,\dots,j\}$
at $p$ is either of smaller dimension than $V_{j-1}$ or has
fewer irreducible components at $p$. This can be continued until some germ $V_N$ is precisely the point $p$. Then the map
$f=(f_1,\dots, f_N)$ gives meromorphic spreadability at $p$.

\item[(iii)] This is obvious.

\item[(iv)] The manifold $X_1 \cap X_2$ is meromorphically convex by Proposition~\ref{MerCvxProp}(ii). The other properties are straight forward.

\item[(v)] This follows from Proposition~\ref{MerCvxProp}(iii).\qedhere
\end{enumerate}
\end{proof}

Many interesting questions remain open concerning {\calligra M }-manifolds. For example, it would be interesting to establish a connection between meromorphic convexity and existence of plurisubharmonic functions or smooth exhaustion functions satisfying additional properties.
A fundamental property of Stein manifolds is that they can be properly embedded into $\mathbb C^N$ for some $N>0$, i.e., a Stein manifold is  biholomorphically equivalent to a closed submanifold of $\mathbb C^N$. In analogy with bimeromorphic equivalence of Moishezon manifolds to projective manifolds, in the context of open {\calligra M\,-\,}manifolds, perhaps, the corresponding property is bimeromorphic equivalence to Stein manifolds or closed submanifolds of $\mathbb C^N$.
For projective manifolds, rational convexity is defined using positive divisors, which are naturally related to positive line bundles. This suggests a connection between meromorphic convexity on an {\calligra M }-manifold $X$  and its Picard group $\text{Pic}(X)$.
Some of these questions will be addressed in our forthcoming work.

\section{Oka--Weil-type theorems on non-Stein manifolds}\label{s.ow}

The classical Oka--Weil theorem states that any holomorphic function on a neighborhood of a polynomially (resp. rationally) convex compact
$K \subset \mathbb C^n$ can be approximated uniformly on $K$ by entire (resp. rational) functions. In this section we give variations of this result for holomorphically and meromorphically convex manifolds.

We say that a compact set $K\subset X$ is \textit{convex with respect to principal hypersurfaces} if its hull
\[
	H(K):=\left\{z\in X\,:\,\forall f\in\mathscr{O}(X)\text{ with }f(z)=0\text{ we have }f^{-1}(0)\cap K\neq\varnothing\right\}
\]
coincides with $K$.

The following result is a meromorphic version of \cite[Theorem 3.2]{Mo20}.

\begin{thm}[Oka--Weil]\label{easyOW}
	Let $X$ be a holomorphically convex manifold and let $K$ be a compact set with $H(K)=K$. Let $U$ be a neighborhood of $K$ on which is defined a holomorphic function $f$. Then for all $\varepsilon>0$, there exist $u,v\in\mathscr{O}(X)$, coprime at each point of $X$, with the property that
\[
	\sup_{z\in K}\left|f-\frac{u}{v}\right|<\varepsilon.
\]
\end{thm}
The proof of Theorem~\ref{easyOW} relies on a so-called Remmert reduction of $X$. This can be constructed as follows: we say that $x$ and $y$ are equivalent if $f(x)=f(y)$ for all $\mathscr{O}(X)$ and we call this relation ``$\sim$'', then, by a theorem of Cartan, $Y:=X\setminus\sim$ is a complex analytic space. If $\varphi:X\to Y$ is the quotient map, then we have $\varphi_*\mathscr{O}(X)=\mathscr{O}(Y)$. It is clear that $Y$ remains holomorphically convex, and, contains no compact complex varieties (since all compact complex varieties in $X$ are identified as points). Therefore $Y$ is a Stein space whose holomorphic functions are identified with those of $X$ in a very natural way, see~\cite[Thm 57.11]{KaKa}.

\begin{proof}[Proof of Theorem~\ref{easyOW}]
	Let $K=H(K)$ be a compact set and let $U$ be a neighborhood of $K$ on which is defined a holomorphic function $f$. First, it must be noted that, if $V$ is a compact irreducible complex variety in $X$, then either $V\subset K$ or $V\subset X\setminus K$. This is because, if $V\not\subseteq K$, then there exists a point $z\in V\setminus K$. The fact that $H(K)=K$ implies that there exists a member $h$ of $\mathscr{O}(X)$ whose zero set passes through $z$ and avoids $K$---but $h(z)=0$ implies that $h|_V\equiv 0$ by the maximum principle. This implies that $V\subset X\setminus K$.

	Let ``$\sim$'' be the aforementioned equivalence relation, and let $Y=X\setminus\!\!\sim$ be the associated Remmert reduction with projection $\varphi: X\to Y$.

	The compact set $\varphi(K)\subset Y$ is convex with respect to principal hypersurfaces. Indeed, suppose that $p\in Y\setminus \varphi(K)$ belongs to $H(\varphi(K))$. Then every $g\in\mathscr{O}(Y)$ with $g(p)=0$ has zero set intersecting $\varphi(K)$. Choose $r\in X$ with $\varphi(r)=p$; we necessarily have $r\not\in K$. It follows that $g\circ\varphi\in\mathscr{O}(X)$ is zero at $r$ and has zero set passing through $K$. Since $\varphi_{*}\mathscr{O}(X)=\mathscr{O}(Y)$, it follows that $r\in H(K)\setminus K$, contrary to $H(K)=K$.

	Applying the Oka--Weil theorem from the previous work of the authors~\cite[Theorem 2.1]{BoSh25} (see the remark following the proof) shows that there exists $u,v\in\mathscr{O}(Y)$, which are pairwise coprime at every point of $Y$, so that \[ \sup_{z\in\varphi(K)}\left|\varphi_*f(z)-\frac{u(z)}{v(z)}\right|<\varepsilon
\]
	this shows that $\frac{u\circ\varphi}{v\circ\varphi}$ is the desired meromorphic function.
\end{proof}

\noindent \textbf{Remark.} Theorem 2.1 from~\cite{BoSh25} is stated for Stein manifolds, however the result is also true for Stein spaces generally. Indeed for the proof to go through, one requires the following results in the context of Stein spaces:

\begin{result}
	If $X$ is a Stein space and $h\in\mathscr{O}(X)$, then $X\setminus h^{-1}(0)$ is a Stein space.
\end{result}
\begin{proof}
This follows from~\cite[Proposition 51.8]{KaKa}.
\end{proof}

\begin{result}[Oka--Weil]
If $X$ is a Stein space and $K$ is a compact holomorphically convex subset of $X$, then every holomorphic function in an open neighborhood of $K$ can be approximated uniformly on $K$ by functions in $\mathscr{O}(X)$.
\end{result}
\begin{proof}
This follows from~\cite[Theorem 2.3.1]{Fo2017}.
\end{proof}
\begin{result}
Given a Stein space $X$, there exists a holomorphic map from $X$ into complex Euclidean space which is injective and proper.
\end{result}
\begin{proof}
This follows from~\cite[p.~127]{GrRe04}.
\end{proof}
For {\calligra M }-manifolds we have a result with some additional assumptions on the compact $K$.

\begin{thm}\label{OW}
Let $X$ be an {\calligra M}-manifold and let $K\subset X$ be a compact set with $\widehat{K}_X=K$. We additionally assume a strengthened versions of conditions (b) and (c) in the definition of an {\calligra M}-manifold; that is, we assume:
\begin{enumerate}
	\item[(b)$'$] for each $p\in K$, there exist $f_1,\ldots,f_n\in\mathscr{M}(X)\cap\mathscr{O}(K)$ that form local coordinates of $X$ near $p$, and

	\item[(c)$'$]for each $p,q\in K$, there exist a $f\in\mathscr{M}(X)\cap\mathscr{O}(K)$ with $f(p)\neq f(q)$.
\end{enumerate}
	Then for any $\varphi\in\mathscr{O}(K)$ and $\varepsilon>0$ there exists $g\in\mathscr{M}(X)$ such that $\|\varphi-g\|_K<\varepsilon$.
\end{thm}

Note that in the context of approximation of holomorphic functions both (b)$'$ and (c)$'$ are natural assumptions on $K$ as can be seen by approximating coordinate functions on a small closed ball $K$ in a coordinate chart on $X$. However, these are not necessary conditions for approximation, e.g., $K=\mathbb{CP}^1\times\overline{\mathbb{D}}\subset X$, where $X=\mathbb{CP}^1\times\mathbb{C}$ and $\mathbb{D}$ denotes the unit disk in $\mathbb{C}$.

\begin{proof}[Proof of Theorem~\ref{OW}]
	Let $U$ be an open neighborhood of $K$ on which $\varphi$ is defined and holomorphic.

	Fix a point $p$ on the topological boundary of $U$. Then, because $p\not\in K=\widehat{K}_X$, there exists a meromorphic function $m\in\mathscr{M}(X)$ with $m\in\mathscr{O}(K\cup\{z\})$ so that $|m(p)|>\|m\|_K$. Choose $\alpha>0$ so that $|m(p)|>\alpha >\|m\|_K$. By replacing $m$ with $m/\alpha$, we can assume that $\|m\|_K<1$ and $|m(p)|>1$, and hence that $|m|>1$ in a neighborhood of $p$. Compactness then ensures the existence of $m_1,\ldots,m_N\in\mathscr{M}(X)$ so that
\[
	K\subset\bigcap_{j=1}^{N}\{z\in X:|m_j(z)|<1\}\subset\subset U.
\]

	Now, define a meromorphic map $\Psi:X\to\mathbb{CP}^1 \times \dots \times \mathbb{CP}^1$ by
\[
	\Psi(z)=\big(m_1(z),\ldots,m_N(z)\big).
\]
	Restricted to $\Pi=\cap_{j=1}^{N}\{z:|m_j(z)|<1\}$, $\Psi$ is a proper holomorphic map into $\mathbb{D}^N$, the unit polydisk of $\mathbb{C}^N$.

	We claim that, by adding more meromorphic functions to $\Psi$ if necessary, we can assume further that $\Psi$ is an embedding on $\Pi$. First, we will show that $\Psi$ can be modified to an immersion on $K$. Accordingly, suppose that there exists a point $z_0$ so that the total derivative $D\Psi$, viewed as a matrix in local coordinates near $z_0$, does not have full rank at $z_0$. We invoke assumption (b)$'$, yielding $f_1,\ldots, f_n\in\mathscr{M}(X)\cap\mathscr{O}(K)$ which form local coordinates near $z_0$. By rescaling, we can assume that $\|f_j\|_K<1$ for all $j$. Then the assignment
\[
	z\mapsto \big(m_1(z),\ldots,m_N(z),f_1(z),\ldots,f_n(z)\big)
\]
	is a meromorphic map on $X$ which is a proper holomorphic map into $\mathbb{D}^{N+n}\subset\mathbb{C}^{N+n}$ when restricted to the set
\[
	\left(\bigcap_{j=1}^{N}\{z\in X:|m_j(z)|<1\}\right)\cap\left(\bigcap_{j=1}^{n}\{z\in X:|f_n(z)|<1\}\right)\subset\subset U,
\]
	and has the additional property that its differential has full rank near $z_0$. This process will terminate after finitely many iterations, in view of compactness of $K$. This means that, after adding finitely many functions, $\Psi$ can be made into an immersion on $K$. Similarly, if $z',z''\in K$ are points such that $\Psi(z')=\Psi(z'')$, we invoke (c)$'$ to find $g\in\mathscr{M}(X)\cap\mathscr{O}(K)$ with $g(z')\neq g(z'')$. We likewise rescale and add $g$ to the map $\Psi$ in order to obtain a meromorphic map with all the same properties as before but additionally attains distinct values near $z'$ and $z''$. Compactness of the set $\{(z,w)\in K\times K:\Psi(z)=\Psi(w)\}$ also ensures this process will terminate after finitely many steps. This shows that $\Psi$ can be made into an embedding on $K$, and hence in a neighborhood $V$ of $K$. By adding yet more meromorphic functions if necessary (repeating the same argument at the beginning of the proof, this time to the topological boundary of $V$), we can assume that $\Pi\subset\subset V$, as well. This proves the claim.

	The map $\Psi$ thus embeds $\Pi$ onto a complex complex subvariety of the unit polydisc $\mathbb{D}^N\subset\mathbb{C}^N$, and hence there exists a $h\in\mathscr{O}(\Psi(\Pi))$ such that $h\circ\Psi=\varphi$. In view of the Oka--Cartan extension theorem~(see, e.g.,~\cite[Corollary 2.6.3]{Fo2017};~\cite{Se53}), we extend $h$ to a function $H\in\mathscr{O}(\mathbb{D}^N)$. Expanding $H$ into a power series and precomposing its Taylor polynomials by $\Psi$ gives a sequence of meromorphic functions converging to $\varphi$ uniformly on $K$.
\end{proof}


\section{A Sufficient Condition for a Long $\mathbb{C}^2$ to Be an {\calligra M }-Manifold}\label{s.longc2}
\begin{definition}
We say that an $n$-dimensional complex manifold $X$ is a \textit{long }$\mathbb{C}^n$ if there is a countable sequence $\{X_j\}_{j}$ of open subsets of $X$ with the following properties:
\begin{enumerate}
	\item[(i)] $X_j\subseteq X_{j+1}$ for all $j$;
	\item[(ii)] each $X_j$ is biholomorphic to $\mathbb{C}^n$; and
	\item[(iii)] $\bigcup_{j}X_j=X$.
\end{enumerate}
\end{definition}
It is not true in general that every long $\mathbb{C}^n$ is biholomorphic to $\mathbb{C}^n$. In fact, Boc Thaler and
Forstneri\v{c}~\cite{BTFo16} have constructed a long $\mathbb{C}^2$ which admits no nonconstant holomorphic functions (see also Wold~\cite{Wo10}).

\begin{definition}
Let $X_1 \subset X_2$ be Stein manifolds of the same dimension. We say that $(X_1, X_2)$ is a {\it meromorphically
Runge pair}, if  for any compact set $K\subset X_1$ and any function $f\in \mathscr M(X_1)$ that is holomorphic in $K$,
and any $\varepsilon>0$, there exists a meromorphic function $g \in \mathscr M(X_2)$ such that $||g - f||_K < \varepsilon$.
\end{definition}

For a related notion for domains in Euclidean spaces, see~\cite{Di07}.

\begin{prop}\label{rungePair}
Let $X_1$ and $X_2$ be Stein manifolds, $X_1 \subset X_2$. The following are equivalent.

\begin{enumerate}
\item[(i)] $(X_1, X_2)$ is a meromorphically Runge pair.

\item[(ii)] $\widehat K_{X_1} = \widehat K_{X_2}$ for any compact $K\subset X_1$.

\item[(iii)] If $K\subset X_1$ is a compact set satisfying $K=\widehat K_{X_1}$, then any holomorphic function in an neighborhood of $K$
can be approximated uniformly on $K$ by meromorphic functions on $X_2$.
\end{enumerate}
If, in addition, $H^2(X_1,\mathbb{Z})=0$, then the above are equivalent to
\begin{enumerate}
\item[(iv)] $\widehat{K}_{X_2}\subset X_1$ for all compact sets $K\subset X_1$.
\end{enumerate}
\end{prop}

\begin{proof}
(i) $ \Longrightarrow$ (ii). Let $K$ be a compact set in $X_1$. Then $\widehat K_{X_1} \subset \widehat K_{X_2}$ as follows from the
inclusion $\mathscr M(X_2) \subset \mathscr M(X_1)$ (as algebras on $X_1$). We need to prove the other inclusion,
which is equivalent to showing that $(\widehat K_{X_1})^c \subset (\widehat K_{X_2})^c$. Suppose that
$p\in X_1 \setminus\widehat K_{X_1}$. Then there exists $f\in \mathscr M(X_1)\cap \mathscr O(K \cup\{p\})$
such that $|f(p)|> ||f||_K$.
	Since $(X_1, X_2)$ is a meromorphically Runge pair, there exists a  function $g \in\mathscr{M}(X_2)$ that approximates $f$ on
$K\cup\{p\}$ well enough so that we have $||g||_K<|g(p)|<\infty$, in particular, $g$ is holomorphic on $K\cup\{p\}$.
And this shows that $p \notin \widehat K_{X_2}$. Since $X_1$ is Stein, $\widehat K_{X_1}$ is compact, and we
	conclude that $\widehat K_{X_1}$ is a connected component of $\widehat K_{X_2}$. It is well-known (c.f. \cite[Corollary 1.5.5]{St08} using \cite[Theorem 2.1]{BoSh25} or~\cite[Theorem 2]{Hi71})
that on a Stein manifold, a connected component of a meromorphically convex compact is itself meromorphically convex,
and this shows that $\widehat K_{X_1} = \widehat K_{X_2}$.

(ii) $ \Longrightarrow$ (iii). This follows from the Oka--Weil theorem.

	(iii) $ \Longrightarrow$ (i). Let $K\subset X_1$ be an arbitrary compact. Since $X_1$ is Stein, $\widehat K_{X_1}$ is compact in $X_1$. If $f$ is a meromorphic function on $X_1$ that is holomorphic on $K$, then by Proposition~\ref{BasicProps}(iii) and Proposition~\ref{steinProject}(i), $f$ is holomorphic on $\widehat K_{X_1}$. By assumption, any holomorphic function on $\widehat K_{X_1}$ can be approximated by meromorphic function on $X_2$, in particular $f$ can be approximated on $K$.

	Now assume $H^2(X_1,\mathbb{Z})=0$. Since (ii) trivially implies (iv), in order to complete the proof it suffices to show that (iv) implies (i). Fix a compact set $K\subset X_1$ and let $f\in\mathscr{M}(X_1)$ be holomorphic on $K$. Since $H^2(X_1,\mathbb{Z})=0$ we can solve a multiplicative Cousin problem~\cite[Proposition V.1.8]{FrGr02} to write $f=u/v$, where $u,v\in\mathscr{O}(X_1)$ with $\gcd(u,v)=1$ and $v\neq 0$ on $K$. Since $\widehat{K}_{X_2}\subset X_1$, the Oka--Weil theorem allows us to approximate both $u$ and $v$ by members of $\mathscr{M}(X_2)$ in the uniform norm on $\widehat{K}_{X_2}$. Given $\varepsilon>0$, we take the quotient of functions sufficiently close to $u$ and $v$, yielding a meromorphic function $g\in\mathscr{M}(X_2)$ with $\|f-g\|_{K}<\varepsilon$.
\end{proof}

The main result of this section is analogous to Theorem 1.2 in~\cite{Wo10}.
\begin{thm}\label{MeroRunge}
    If $X=\bigcup_{j=0}^{\infty}X_j$ is a long $\mathbb{C}^2$ and $(X_j,X_{j+1})$ is a meromorphically Runge pair for each $j$, then $X$ is an {\calligra M }-manifold.
\end{thm}

While holomorphically uninteresting, the manifold $X=\bigcup_{j=0}^{\infty}X_j$---the long $\mathbb{C}^2$  constructed by Boc Thaler and Forstneri\v{c}---has the property that a holomorphically convex compact set $K$ in $X_j$ is rationally convex when viewed as a compact set in $X_{j+1}$~\cite[p. 5]{BTFo16}. Since $X_j\cong\mathbb{C}^2$, this implies that $(X_j, X_{j+1})$ is a meromorphically Runge pair by Proposition~\ref{rungePair}(iv) and hence $X$ satisfies the hypotheses of Theorem~\ref{MeroRunge}.

Define the spherical metric on $\mathbb{CP}^1\cong\mathbb C \cup \{\infty\}$ as follows: For two complex points $z,w\in\mathbb{C}$,
\[
	|w,z|=\frac{|w-z|}{\sqrt{1+|w|^2}\sqrt{1+|z|^2}},
\]
while
\[
	|w,\infty|=|\infty,w|=\frac{1}{\sqrt{1+|w|^2}}\qquad\text{and}\qquad |\infty,\infty|=0.
\]

For meromorphic functions without indeterminacy points, convergence of $f_j\to f$ in the spherical metric means precisely that every point admits a neighborhood on which $f_j\to f$ or $1/f_j\to 1/f$ converges uniformly.

\begin{proof}[Proof of Theorem~\ref{MeroRunge}]
	We will first show that $X$ is meromorphically convex. Let $K\subset X$ be compact. Then there exists some $X_k$ in $X=\bigcup_{j=0}^{\infty} X_j$ for which $K\subset X_j$. By relabeling the indices on the collection $\{X_j\}$ and omitting finitely many members if necessary, we can assume $k=0$. Since $X_j$ is meromorphically Runge in $X_{j+1}$ for each $j\geq 0$, we have
\begin{equation}\label{MRunge}
		\widehat{K}_{X_0}=\widehat{K}_{X_j}\quad\text{for all $j\geq 0$.}
\end{equation}

	It is sufficient to show that $\widehat{K}_X\subset\widehat{K}_{X_0}$. Accordingly, fix a point $p\in X\setminus\widehat{K}_{X_0}$. Taking into account~\eqref{MRunge}, we can again relabel if necessary to assume $\{p\}\cup K\subset X_0$ while retaining $p\in X\setminus\widehat{K}_{X_0}$.

	By Proposition~1.2 in~\cite{BoSh25}, there exists a meromorphic function $m_0\in\mathscr{M}(X_0)$ with $m_0\in\mathscr{O}(K)$, $p\not\in\mathcal{I}(m_0)$, and $|m_0(p)|>\|m_0\|_K+\delta$ for some small $\delta>0$. Fix an increasing sequence of nested closed sets $\{\overline{B}_j\}_{j=0}^{\infty}$ of $X$ such that
\begin{itemize}
	\item Each $\overline{B}_j$ is a closed ball in $X_j$ when viewed through the given biholomorphism $X_j\to\mathbb{C}^2$;
	\item $\overline{B}_j$ is compact in $B_{j+1}$ for each $j$;
	\item $\widehat{K}_{X_0}\cup\{p\}\subset B_0$; and
	\item $X=\bigcup_{j=0}^{\infty}\overline{B}_j$.
\end{itemize}
Since $X\cong\mathbb{C}^2$, we can assume, without loss of generality, that $m_{0}$ is rational. Consequently, there exist polynomials $f_0,g_0\in\mathscr{O}(X_0)$ with $\gcd(f_0,g_0)=1$ such that $m_0=f_0/g_0$; that is, $m_0$ is the quotient $f_0$ and $g_0$ with $Z(m_0)=Z(f_0)$ and $P(m_0)=Z(g_0)$. Furthermore, by perturbing $f_0$ and $g_0$ if necessary, we can assume that $Z(f_0)$ and $Z(g_0)$ have only transverse intersection within $B_0$---in particular this means that the intersection multiplicity (c.f.~\cite[Lecture 18]{Ha95}) of $Z(f_0)$ and $Z(g_0)$ at these points is one. Write $\mathcal{I}(m_0)\cap\overline{B}_0=\{s^0_j\}_{j=0}^{N_0}$ and choose a large positive integer $\ell_0$ so that
\begin{itemize}
	\item $\overline{\mathbb{B}}_0(s^0_j,2^{-\ell_0})\cap\overline{\mathbb B}_0(s^0_k,2^{-\ell_0})=\varnothing$ for all $j\neq k$, where $\mathbb{B}_0(q,r)\subset X_0\cong\mathbb{C}^2$ is the (open) ball centered at $q\in X_0$ with radius $r>0$.
	\item $\bigcup_{j=0}^{N_0}\overline{\mathbb B}_0(s_j^0,2^{-\ell_0})\cap\big(\widehat{K}_{X_0}\cup\{p\}\big)=\varnothing$; and
	\item $\overline{\mathbb{B}}_0(s^0_j,2^{-\ell_0})$ contains a portion of precisely one irreducible component of each of $Z(f_0)$ and $Z(g_0)$ for each $j$.
\end{itemize}

$\overline B_0$ is rationally convex in $X_0\cong\mathbb{C}^2$, so~\eqref{MRunge} implies that $\overline{B}_0$ is rationally convex in $X_1\cong\mathbb{C}^2$ as well. In view of the Oka--Weil Theorem~\cite[p. 44]{St08}, $f_0$ and $g_0$ can each be approximated uniformly on $\overline{B}_0$ by rational functions on $X_{1}\cong \mathbb{C}^2$. Accordingly, choose a rational function $m_1\in\mathscr{M}(X_1)$ with
\begin{itemize}
	\item $m_1=f_1/g_1$ for rational $f_1,g_1\in\mathscr{M}(X_1)\cap\mathscr{O}(\overline{B}_0)$;
	\item $\mathcal{I}(m_1)\cap\overline{B}_0\subset\bigcup_{j=0}^{N_0}\overline{\mathbb{B}}_0(s^0_j,2^{-\ell_0})$;
	\item the inequality
\[
	\sup_{w\in\overline{B}_0\setminus\cup_{j=0}^{N_0}\overline{\mathbb{B}}_0(s^0_j,2^{-\ell_0})}|m_1(w),m_0(w)|<1
\]
is satisfied;
\item the inequality
\[
	|m_1(p)|>\|m_1\|_K+\delta
\]
persists; and
\item The hypersurfaces $Z(f_1)$ and $Z(g_1)$ have an intersection multiplicity of one at each of their points of intersection within $\overline{B}_0$.
\end{itemize}

	The last point requires some care.  The hypersurfaces $Z(f_0)$ and $Z(g_0)$ intersect transversely within $\overline{B}_0$, implying that their complex gradients are linearly independent at such points. Normal convergence of holomorphic functions implies normal convergence of their derivatives to the respective derivatives of the limiting functions, so $Z(f_1)$ and $Z(g_1)$ can be chosen to have transverse intersection at these points as well.

	Now, fix $s^0_j\in\mathcal{I}(m)$ and let $d_1$ and $d_2$ be the orders of vanishing of $f_0$ and $g_0$, respectively. Hurwitz's theorem~\cite[Theorem VII.2.5]{Co78} applied to the intersection of complex lines with $Z(f_0)$ and $Z(g_0)$ near $s^0_j$ shows that $Z(f_1)$ and $Z(g_1)$ have $d_j$ and $e_j$ irreducible components (counting multiplicity), respectively, within $\overline{\mathbb{B}}_0(s^0_j,2^{-\ell_0})$. It follows that $m_1=f_1/g_1$ has at most $d_j\cdot e_j$ indeterminacy points, all having intersection multiplicity one, within $\overline{\mathbb{B}}_0(s^0_j,2^{-\ell_0})$.

	Proceeding, we write $\mathcal{I}(m_1)\cap\overline{B}_1=\{s^1_j\}^{N_1}_{j=0}$, and choose an integer $\ell_1>\ell_0$ so that
\begin{itemize}
	\item $\mathbb{B}_1(z,2^{-\ell_1})\subset\mathbb{B}_0(z,2^{-(\ell_0+1)})$ for $z\in\overline{B}_0$, where $\mathbb{B}_1(q,r)\subset X_1\cong\mathbb{C}^2$ now denotes a ball in $X_2\cong\mathbb{C}^2$;
	\item $\overline{\mathbb{B}}_1(s^1_j,2^{-\ell_1})\cap\overline{\mathbb B}_1(s^1_k,2^{-\ell_1})=\varnothing$ for all $j\neq k$;
	\item $\bigcup_{j=0}^{N_1}\overline{\mathbb B}_1(s_j^1,2^{-\ell_1})\cap\big(\widehat{K}_{X_0}\cup\{p\}\big)=\varnothing$; and
	\item $\overline{\mathbb{B}_1}(s^1_j,2^{-\ell_0})$ contains a portion of precisely one irreducible component of each of $Z(f_1)$ and $Z(g_1)$ for each $j$.
\end{itemize}
	We argue as before to find a rational function $m_2\in\mathscr{M}(X_2)$ such that
\begin{itemize}
	\item $m_2=f_2/g_2$ for rational $f_2,g_2\in\mathscr{M}(X_2)\cap\mathscr{O}(\overline{B}_1)$;
	\item $\mathcal{I}(m_2)\cap\overline{B}_1\subset\bigcup_{j=0}^{N_1}\overline{\mathbb{B}}_1(s^1_j,2^{-\ell_1})$;
	\item the inequality
\[
	\sup_{w\in\overline{B}_1\setminus\cup_{j=0}^{N_1}\overline{\mathbb{B}}_1(s^1_j,2^{-\ell_1})}|m_2(w),m_1(w)|<\frac{1}{2}
\]
is satisfied;
\item the inequality
\[
	|m_2(p)|>\|m_2\|_K+\delta
\]
persists; and
\item The hypersurfaces $Z(f_2)$ and $Z(g_2)$ have an intersection multiplicity of one at each of their points of intersection within $\overline{B}_1$.
\end{itemize}

Once again, Hurwitz's theorem asserts that the number of indeterminacy points of $m_2$ within $\overline{B}_0$ is bounded above by $\sum_{j=0}^{N_0}d_j\cdot e_j$. In particular, this implies that after finitely many steps in the inductive process to follow, the number of indeterminacy points contained within $\overline{B}_0$ will stabilize.

We proceed inductively to construct a sequence $m_0,m_1,m_2,\ldots$ of meromorphic functions with respective indeterminacy sets $\{s^0_j\}_{j=0}^{\infty},\{s^1_j\}_{j=0}^{\infty},\{s^2_j\}_{j=0}^{\infty},\ldots$ and a sequence of positive integers $\ell_0<\ell_1<\ell_2<\ldots$ having the following properties for each $k$:
\begin{enumerate}
	\item[(i)] $m_{k+1}=\frac{f_{k+1}}{g_{k+1}}$ for rational $f_{k+1},g_{k+1}\in\mathscr{M}(X_{k+1})\cap\mathscr{O}(\overline{B}_{k})$;
	\item[(ii)]$\mathbb B_{k+1}(z,2^{-\ell_{k+1}})\subset\mathbb{B}_k(z,2^{-(\ell_k+1)})$ for $z\in\overline{B}_k$, where $\mathbb{B}_k(q,r)$ is a ball in $X_k\cong\mathbb{C}^2$;
	\item[(iii)] $\mathcal{I}(m_{k+1})\cap\overline{B}_k\subset\bigcup_{j=0}^{N_k}\overline{\mathbb{B}}_k(s^k_j,2^{-\ell_k})$;
	\item[(iv)] the inequality
	\[
		\sup_{w\in\overline{B}_{k}\setminus\cup_{j=0}^{N_k}\overline{\mathbb{B}}_k(s^k_j,2^{-\ell_k})}|m_{k+1}(w),m_k(w)|<\frac{1}{2^k}
	\]
		holds for each $k$;
	\item[(v)] the inequality
\[
	|m_k(p)|>\|m_k\|_K+\delta
\]
holds for each $k$.
\end{enumerate}

		Fix an integer $t\geq 0$. For large $k$ the number of indeterminacy points of $m_k$ within $\overline{B}_t$ will not change. Thus, for a fixed point of $a_k\in\mathcal{I}(m_k)$ within $\overline{B}_t$, $k$ large, there is a nearby indeterminacy point $a_{k+1}$ of $m_{k+1}$, which, in turn, has a nearby indeterminacy point $a_{k+2}$ of $m_{k+2}$, and so on, yielding a sequence $a_k,a_{k+1},\ldots$. This sequence is Cauchy due to (ii) and (iii), and hence converges to a point $a$. This process shows that $\mathcal{I}(m_k)$ converges in the Hausdorff metric on compact sets to a countable set $S$.

		Furthermore, (iv) shows that $\{m_k\}_{j=0}^{\infty}$ is uniformly Cauchy in the spherical metric on any compact set avoiding $S$. Therefore, viewing the $m_k$ as holomorphic maps from $X\setminus S$ into $\mathbb{CP}^1$, there exists a holomorphic map $m:X\setminus S\to\mathbb{CP}^1$ to which the sequence $\{m_k\}_{k=0}^{\infty}$ converges locally uniformly in the spherical metric. $m$ is thus a meromorphic function on $X\setminus S$, and a result of Chirka~\cite{Ch96} implies that $m$ extends meromorphically to all of $X$. Finally,
\[
	|m(p)|>\|m\|_k
\]
holds due to (v), showing $\widehat{K}_X\subset\widehat{K}_{X_0}$, as desired, showing that $X$ is meromorphically convex.

We next show that condition (b) of Definition~\ref{M-manifold} is satisfied. Choose two distinct points $p,q\in X$. Then there exists a $j$ so that $p,q\in X_j$ and a holomorphic function $h$ on $X_j\cong\mathbb{C}^2$ with $h(p)\neq h(q)$. Through the same process as above, $h$ can be approximated uniformly on the compact set $\{p\}\cup\{q\}$ by members of $\mathscr{M}(X)$. Taking a sufficiently close approximant of $h$ shows that $\mathscr{M}(X)$ separates points.

That condition (c) of Definition~\ref{M-manifold} is satisfied follows a similar argument. Indeed, given a point $p\in X$, we choose a $j$ so that $p\in X_j\cong \mathbb{C}^2$. The coordinates on $\mathbb{C}^2$ give rise to holomorphic functions $h_1,h_2$ on $X_j\subset X$ which serve as local coordinates near $p$. We now approximate $h_1$ and $h_2$ by meromorphic functions on $X$. Since the Jacobian of the map $(h_1,h_2)$ is nonzero near $p$, and normal convergence of holomorphic functions implies normal convergence of their derivatives, the Jacobian of the approximating meromorphic map will also be nonzero near $p$ for a sufficiently close approximant.
\end{proof}

\section{Further results}

In this section we prove some additional results on the structure of meromorphic functions for certain classes of  complex manifolds which are not necessarily Stein.

Recall that a complex manifold $X$ is \textit{1-convex} if $X$ admits a smooth plurisubharmonic exhaustion function $\varphi:X\to\mathbb{R}$ that is strictly plurisubharmonic outside of a compact set $K\subset X$. Equivalently, $X$ is 1-convex if $X$ is holomorphically convex and there exists a compact set $K$ containing all compact analytic varieties of positive dimension in $X$. The smallest such $K$ is called the \textit{exceptional set} of $X$.

It is well known that a holomorphically convex manifold $X$ is Stein if and only if $X$ admits no compact analytic varieties of positive
dimension, see, e.g.,~\cite[Ch. V, Thm 3.1]{FrGr02}.
The first result of this section is to show an analogue of this phenomenon for meromorphically spreadable manifolds. Playing the role of compact analytic varieties in this setting will be compact \textit{meromorphically trivial} varieties, defined as follows:

\begin{definition}
	We say that a complex space $X$ is \textup{meromorphically trivial} if $X$ admits no nonconstant meromorphic functions, that is, if $\mathscr{M}(X)=\mathbb{C}$.
\end{definition}

The existence of such manifolds was discussed in Section~\ref{s.mm}.

\begin{thm}\label{noMeroVarieties}
	Let $X$ be a connected 1-convex manifold which contains no compact, meromorphically trivial analytic subsets of positive dimension. Then $X$ is meromorphically spreadable.
\end{thm}

The proof requires a proposition.

\begin{prop}
	Let $X$ be a 1-convex complex manifold and suppose that $A\subset X$ is an irreducible analytic variety with the property that $m|_{A}$ is constant for every $m\in\mathscr{M}(X)$. Then $A$, viewed as a complex space, is meromorphically trivial.
\end{prop}

\begin{proof}
By contraposition, it suffices to show that if $A$ admits a nonconstant meromorphic function, then there exists a meromorphic function on $X$ which restricts to a nonconstant meromorphic function on $A$. Let $m\in\mathscr{M}(A)$ be a nonconstant meromorphic function. Let $\mathcal{J}_A$ be the sheaf of meromorphic functions which vanish on $A$. Consider the exact sequence
	\[
		0\longrightarrow\mathcal{J}_A\longrightarrow\mathscr{M}_X\longrightarrow\mathscr{M}_X/\mathcal{J}_A\longrightarrow 0.
	\]
	This induces the exact sequence of \v{C}ech cohomology groups
	\[
		0\longrightarrow H^0(X,\mathcal{J}_A)\longrightarrow H^0(X,\mathscr{M}_X)\longrightarrow H^0(X,\mathscr{M}_X/\mathcal{J}_A)\xlongrightarrow{\delta} H^1(X,\mathcal{J}_A).
	\]
	Since $X$ is 1-convex, the cohomology group $H^1(X,\mathcal{J}_A)$ is finite-dimensional (see Andreotti--Grauert~\cite{AG62}, Narasimhan~\cite{Na61}, Markoe~\cite{Ma81}). Therefore, for some large $N$ the collection $\{m,m^2,\ldots,m^N\}$ will have linearly dependent image through the map $\delta$. This means that there are $a_1,\ldots a_N\in\mathbb{C}\setminus\{0\}$ so that $m':=\sum_{j=1}^{N}a_jm^j$ is in the kernel of $\delta$. Furthermore, $m'$ is nonconstant as well. Indeed, applying Theorem~2.2.9 in~\cite{MaMa} to a single meromorphic function implies that $dm\neq 0$ if and only if $P(m)\neq 0$ for all nontrivial polynomials $P$ of one complex variable. The meromorphic function $m'=Q(m)$, where $Q(z)=\sum_{j=1}^{N}a_jz^j$, is therefore nonconstant.

Because this is an exact sequence, there exists an $M\in\mathscr{M}(X)$ with $M|_A=m'$.
\end{proof}

\begin{proof}[Proof of Theorem~\ref{noMeroVarieties}]
	Let $x_0\in X$ be an arbitrary point. Then the set
\[
	A:=\bigcap_{\substack{f\in\mathscr{M}(X)\\ x_0\not\in\mathcal{I}(f)}} \overline{f^{-1}(f(x_0))}
\]
	is a closed analytic subset of $X$. Clearly, it is contained in $\widehat{\{x_0\}}_X$, the meromorphically convex hull of $\{x_0\}$. Since $\widehat{\{x_0\}}$ must be compact, $A$ is compact. Since $A$ has the property $m|_A$ is constant for every $m\in\mathscr{M}(X)$, the proposition above shows that $A$ is meromorphically trivial. This is possible only if $A$ consists of isolated points. Then there exists an open neighborhood $U$ and meromorphic functions $m_1,\ldots,m_n$ in $U$ such that
\[
	\{x_0\}=A\cap U=\left\{x\in U\,:m_1(x)=\ldots =m_n(x)=0\right\}.
\]
	This implies that  $X$ is meromorphically spreadable.
\end{proof}

Let $X$ be a 1-convex complex manifold, and $S$ be its exceptional set, i.e., the union of all compact complex varieties of positive dimension. By passing to a Remmert reduction, it is clear that any meromorphic function that is constant on the irreducible components of $S$ can be represented as a quotient of two entire functions on $X$. The following theorem gives such representation which in addition is (globally) coprime.

\begin{thm}\label{StrongPoincare}
	Let $X$ be a 1-convex complex manifold, and $S$ be its exceptional set. Suppose that $H^2(X,\mathbb{Z})=0$. Then any $m\in\mathscr{M}(X)$ which is constant on the irreducible components of $S$ admits the representation $m=f/g$ where $f,g\in\mathscr{O}(X)$ and $\gcd(f,g)=1$.
\end{thm}
We require a lemma.
\begin{lem}
	Let $D$ be an effective divisor on a 1-convex complex manifold $X$ which has $H^2(X,\mathbb{Z})=0$. Suppose that $\text{supp}(D)$ avoids the exceptional set $S$ of $X$. Then $D$ is the zero divisor of a holomorphic function on $X$.
\end{lem}
\begin{proof}
	Choose local defining functions $f_j\in\mathscr{O}(U_j)$ over some open cover $\{U_j\}_{j=1}^{\infty}$ of $\text{supp}(D)$. By shrinking the elements of the cover if necessary, we can assume that $\bigcup_{j}U_j$ avoids $S$. Choose an open set $U_0$ so that $\{U_j\}_{j=0}^{\infty}$ is an open cover of $X$ with $S\subset U_0$. Likewise choose $f_0\in\mathscr{O}(U_0)$ to be $f_0\equiv 1$. Then $g_{jk}=f_j/f_k\in\mathscr{O}^*(U_j\cap U_k)$ represents a member of $H^1(X,\mathscr{O}^*)$, the Picard group of $X$.

	On the other hand, let $\mathscr{C}$ and $\mathscr{C}^*$ denote the additive and multiplicative sheaf of germs of complex-valued continuous functions and nonzero continuous functions, respectively. Consider the exact sequence
\[
	0\longrightarrow\mathbb{Z}\longrightarrow\mathscr{C}\xlongrightarrow{\exp(2\pi i(\cdot))}\mathscr{C}^*\longrightarrow 0.
\]
	This induces the exact sequence of \v{C}ech cohomology groups
\[
\ldots\longrightarrow H^1(X,\mathscr{C})\longrightarrow H^1(X,\mathscr{C}^*)\longrightarrow H^2(X,\mathbb{Z})\longrightarrow\ldots.
\]
	Since $H^1(X,\mathscr{C})=0$ for any paracompact space, and $H^2(X,\mathbb{Z})=0$ by assumption, we see that $H^1(X,\mathscr{C}^*)=0$. This means that $\{g_{jk}\}$ represents a trivial cohomology class in $H^1(X,\mathscr{C}^*)$. Therefore there exist $c_j\in\mathscr{C}^*(U_j)$ such that $c_j/c_k=g_{jk}$ in $U_j\cap U_k$ for all $j,k$.

	We now modify $c_0$. Let $V_0$ be an open set with $S\subset V_0\subset\subset U_0\setminus\bigcup_{j=1}^{\infty}U_j$ and let $\chi$ be a smooth function with compact support in $U_0\setminus \bigcup_{j=1}^{\infty}U_j$ that is identically equal to one on $V_0$. Choose $h\in\mathscr{O}(U_0)$ and define $\tilde{c}_0=\chi h+(1-\chi)c_0$. Then $\tilde{c}_0$ is holomorphic in a neighborhood of $S$. We likewise define $\tilde{c}_j=c_j$ for $j=1,2,\ldots$. We see that $\{\tilde{c}_j\}_{j=0}^{\infty}$ is a solution to $\tilde{c}_j/\tilde{c}_k=g_{jk}$ on $U_j\cap U_k$ that is holomorphic in a neighborhood of $S$, so the work of Henkin--Leiterer~\cite{HeLe}
	shows there exist $h_j\in\mathscr{O}^*(U_j)$ such that $h_j/h_k=g_{jk}$ for all $j,k$. Equivalently, $\{g_{jk}\}$ represents the trivial bundle in $H^1(X,\mathscr{O}^*)$. We conclude that $D$ is the zero divisor of a holomorphic function on $X$.
\end{proof}
\begin{proof}[Proof of Theorem~\ref{StrongPoincare}]
	Let $m$ be a meromorphic function which is constant on the irreducible components of $S$. By adding the appropriate constant, we may assume that $m$ is nonzero on $S$. Write $\text{div}(m)=Z(m)-P(m)$, where $Z(m)$ and $P(m)$ denote the (effective) divisors of zeroes and poles of $M$, respectively. Note that $Z(m)$ is an effective divisor whose support avoids $S$, so by the lemma above, $Z(m)$ is principal---that is, there exists a $f\in\mathscr{O}(X)$ with $\text{div}(f)=Z(m)$. Therefore $g:=f/m$ is a holomorphic function on $X$ with $\text{div}(g)=P(m)$, and we conclude that $m=f/g$ is the desired representation.
\end{proof}

At present, no example of a 1-convex manifold $X$ with $H^2(X,\mathbb{Z})=0$ is known. However, there is no known obstruction to the existence of such examples. In particular, Bassanelli and Leoni~\cite{BaLe07} construct 1-convex manifolds whose exceptional set is a closed complex curve that is null-homologous.
\printbibliography

\end{document}